\documentclass[11pt]{article}

\usepackage[paperwidth=8.5in,paperheight=11in,portrait,top=1in,bottom=1.in,left=1.15in,right=1.15in]{geometry}
\usepackage[utf8]{inputenc}
\usepackage{authblk}

\usepackage{amsfonts,amsmath,amssymb,amsthm}
\usepackage{latexsym,mathrsfs,mathtools,bm}
\usepackage{graphicx,subcaption,epsfig,caption,float,xcolor}
\usepackage{enumitem}

\usepackage[hidelinks]{hyperref}
\usepackage{bookmark}

\def\cA{{\mathcal A}}   \def\cB{{\mathcal B}}   
      
      \def\cI{{\mathcal I}}

\newcommand\al{\alpha}
\newcommand\bt{\beta}

\newcommand\vp{\varphi}

\newcommand\wt{\widetilde}

\newcommand\Y{Y} 
\renewcommand\L{L}
\newcommand\M{M}

\theoremstyle{plain}
\newtheorem{thm}{Theorem}[section]

\newtheorem{prop}[thm]{Proposition}

\newtheorem{rem}[thm]{Remark}

\numberwithin{equation}{section}

\title{\bf $R_I$ biorthogonal polynomials of Hahn type}
\renewcommand*{\Affilfont}{\normalsize\small}
\author[1]{Luc Vinet}
\author[2]{Meri Zaimi}
\author[3]{Alexei Zhedanov\vspace{.5em}}
\affil[1,2]{Centre de Recherches Math\'ematiques, Universit\'e de Montr\'eal, P.O. Box 6128, Centre-ville Station, Montr\'eal (Qu\'ebec), H3C 3J7, Canada. \vspace{.5em}}
\affil[1]{Insitut de valorisation des donn\'ees (IVADO), Montr\'eal (Qu\'ebec), H2S 3H1, Canada. \vspace{.5em}}
\affil[3]{School of Mathematics, Renmin University of China, Beijing 100872, China. \newline\vspace{.9em}}

{
	\makeatletter
	\renewcommand\AB@affilsepx{: \protect\Affilfont}
	\makeatother
	\affil[ ]{E-mail addresses}
	\makeatletter
	\renewcommand\AB@affilsepx{, \protect\Affilfont}
	\makeatother
	\affil[1]{luc.vinet@umontreal.ca}
	\affil[2]{meri.zaimi@umontreal.ca}
	\affil[3]{zhedanov@yahoo.com}
}

\begin{document}
	
\date{\today} 
\maketitle

\begin{abstract}
	A finite family of $R_I$ polynomials is introduced and studied. It consists in a set of polynomials of $_{3}F_{2}$ form whose biorthogonality to an ensemble of rational functions is spelled out. These polynomials are shown to satisfy two generalized eigenvalue problems: in addition to their recurrence relation of $R_I$ type, they are also found to obey a difference equation. Underscoring this bispectrality is a triplet of operators with tridiagonal actions.  
\end{abstract}


\maketitle

\section{Introduction}\label{sec:intro}

The purpose of this paper is to study the biorthogonality and bispectrality of a finite family of hypergeometric polynomials of Hahn type which can be viewed as discrete analogues of the biorthogonal Askey polynomials on the unit circle.

Biorthogonal rational functions (BRF's) with bispectral properties provide a natural extension to the classical orthogonal polynomials (OP's) of the ($q$-)Askey scheme. Indeed, all families of classical OP's are bispectral, meaning that these polynomials satisfy both a three-term recurrence relation and a difference or differential equation \cite{Koek}. This bispectrality can be understood algebraically in terms of two eigenvalue problems involving a pair of operators, known as a Leonard pair in the finite-dimensional case \cite{Ter}. The algebra generated by these bispectral operators is, in its most general form, the Askey--Wilson algebra \cite{Zhe91}. The OP's can then be viewed as solutions of two eigenvalue problems posited in terms of these operators or as overlap coefficients when studying the representation theory of the associated algebra.  

The BRF's appear when considering instead generalized eigenvalue problems (GEVP's) of the form $LP(x)=\lambda MP(x)$, where $L$ and $M$ are two operators which both act tridiagonally in a certain basis \cite{Zhe99}. The bispectrality in this case occurs when the BRF's are solutions to two such GEVP's involving a triplet of operators. Different systems of bispectral biorthogonal functions have been studied in previous works. The rational functions of Hahn type are examined in \cite{TVZ}, the representation theory of their associated algebra in \cite{VZ_Hahn} and their $q$-deformation in \cite{BGVZ}. Similarly, in \cite{VZ_Ask}, an algebraic approach is used to understand the biorthogonality and bispectrality of the Askey polynomials on the unit circle. In the present paper, we pursue this program by examining some polynomials of ${_3}F_2$ type which are biorthogonal on a real linear grid. These functions can be viewed as a discrete version of the Askey polynomials on the unit circle, as will be explained. They will be called $R_I$ polynomials of Hahn type in view of the recurrence relation they satisfy.                   

The paper is organized as follows. Section \ref{sec:tripop} introduces three difference operators that act in a tridiagonal fashion on a certain polynomial basis. A GEVP involving two of these operators is considered in Section \ref{sec:GEVP}. Its solutions define a finite set of polynomials expressed as ${_3}F_2$ hypergeometric functions, the $R_I$ polynomials of Hahn type. The biorthogonality relation for these polynomials is obtained in Section \ref{sec:biorth} through the adjoint GEVP. The biorthogonal partners are explicitly given by ${_3}F_2$ rational functions. The bispectral properties of the $R_I$ polynomials of Hahn type are obtained in Section \ref{sec:bispect} using the action of the triplet of operators. Section \ref{sec:HahnOP} establishes a connection between the $R_I$ biorthogonal polynomials of Hahn type and the Hahn classical orthogonal polynomials. This allows to obtain an explicit expression for the normalization constant in the biorthogonality relation of the $R_I$ polynomials. A connection with $_{10} \phi_9$ basic hypergeometric biorthogonal rational functions is also exposed in Section \ref{sec:biorthlim}. Section \ref{sec:conc} contains concluding remarks.         

\section{A triplet of difference operators}\label{sec:tripop}
We start by introducing three operators that will be central in the study of the properties of the $R_I$ biorthogonal polynomials of Hahn type.

Let $N$ be a non-negative integer and $\al,\bt$ some real parameters. We define the following operators acting on the space of polynomials on the linear grid $x=0,1,\dots,N$
\begin{align}
	&\L^{(\al,\bt)}= (N-x)T^+ +(x-N-\al-1)\mathcal{I}, \label{eq:L_op} \\
	&\M^{(\alpha,\beta)}= A(x)T^+ + B(x)T^- + C(x) \mathcal{I} \label{eq:M_op},\\
	&\Y^{(\al,\bt)} = x(N-x)T^+ +x(x-N-\al-1)\mathcal{I}, \label{eq:Y_op}
\end{align}
where
\begin{equation}
	A(x) = (N-x)(x-\beta+1), \quad B(x)=x(\alpha+\beta+N+1-x), \quad C(x) = -(A(x)+B(x)). \label{eq:ABC}
\end{equation}
The notations $T^{\pm}$ and $\cI$ refer respectively to the shift and identity operators which act as  
\begin{equation}
	T^{\pm} f(x) = f(x \pm 1), \quad \cI f(x) = f(x).
\end{equation}
One observes from the definitions that
\begin{equation}
	\Y^{(\al,\bt)} = x \L^{(\al,\bt)}. \label{eq:YxL}
\end{equation}
From now on, we will not write explicitly the dependency of the operators on the parameters $\al,\bt$ except when necessary.

A basis for the space of polynomials of degree less or equal to $N$ is 
\begin{equation}
	\vp_n(x) =(-x)_n \quad \text{for } n=0,1,\dots,N,
	\label{eq:basis_phi}
\end{equation}
where
\begin{equation}
	(a)_0:=1, \quad (a)_n := a(a+1) \dots (a+n-1) \quad \text{for } n=1,2,\dots \label{eq:Pochh}
\end{equation}
is the Pochhammer symbol (or shifted factorial). We will also use the conventions $\vp_{-1}(x)=\vp_{N+1}(x)=0$. The action of the operators \eqref{eq:L_op}--\eqref{eq:Y_op} on the basis \eqref{eq:basis_phi} is, for $n=0,1,\dots,N$,
\begin{align}
	&\L \vp_n(x) = \eta_n^{(1)} \vp_n(x) + \eta_n^{(2)} \vp_{n-1}(x) \label{eq:L_phi}, \\
	&\M \vp_n(x) = \eta_n^{(3)} \vp_n(x) + \eta_n^{(4)} \vp_{n-1}(x) \label{eq:M_phi},\\
	&\Y \vp_n(x) = - \eta_n^{(1)} \vp_{n+1}(x) +\eta_n^{(5)}\vp_n(x) + \eta_n^{(6)}\vp_{n-1}(x), \label{eq:Y_phi}
\end{align} 
where
\begin{alignat}{2}
	&\eta_n^{(1)} = -(n+\al+1), \qquad && \eta_n^{(2)} = n(n-N-1), \\ &\eta_n^{(3)} = -n(n+\al+1), \qquad && \eta_n^{(4)} = n(n-\beta)(n-N-1),\\
	&\eta_n^{(5)} = -n(2n+\al-N), \qquad && \eta_n^{(6)} = n(n-1)(n-N-1).
\end{alignat} 
It follows that these operators are represented by tridiagonal $(N+1)\times (N+1)$ matrices in the basis $\{ \vp_n\}_{n=0}^N $, with $\L$ and $\M$ upper bidiagonal:
\begin{align}
	\L &= 
	\begin{pmatrix}
		\eta_0^{(1)} & \eta_1^{(2)} & 0 & \cdots & 0 \\
		0            & \eta_1^{(1)} & \eta_2^{(2)} & & \vdots \\
		0 & 0 & \eta_2^{(1)} & \ddots & 0 \\ 
		\vdots & & \ddots & \ddots & \eta_N^{(2)} \\
		0 & \cdots & 0 & 0 & \eta_{N}^{(1)}
	\end{pmatrix}, \label{eq:L_mat} \\[0.5em]
	\M &= 
	\begin{pmatrix}
		\eta_0^{(3)} & \eta_1^{(4)} & 0 & \cdots & 0 \\
		0            & \eta_1^{(3)} & \eta_2^{(4)} & & \vdots \\
		0 & 0 & \eta_2^{(3)} & \ddots & 0 \\ 
		\vdots & & \ddots & \ddots & \eta_N^{(4)} \\
		0 & \cdots & 0 & 0 & \eta_{N}^{(3)}
	\end{pmatrix}, \label{eq:M_mat} \\[0.5em]
	\Y &=
	\begin{pmatrix}
		\eta_0^{(5)} & \eta_1^{(6)} & 0 & \cdots & 0 \\
		-\eta_0^{(1)}            & \eta_1^{(5)} & \eta_2^{(6)} & & \vdots \\
		0 & -\eta_1^{(1)} & \eta_2^{(5)} & \ddots & 0 \\ 
		\vdots & & \ddots & \ddots & \eta_N^{(6)} \\
		0 & \cdots & 0 & -\eta_{N-1}^{(1)} & \eta_{N}^{(5)}
	\end{pmatrix}. \label{eq:Y_mat}
\end{align}

\section{$R_I$ polynomials of Hahn type as solutions of a generalized eigenvalue problem}\label{sec:GEVP}

In this section, we show that $R_I$ polynomials of Hahn type arise as solutions of a GEVP involving the operators $\M$ and $\L$. Throughout this paper, the notation ${_r}F_s$ will be used to refer to the standard hypergeometric function, see \cite{GR} for instance.
\begin{prop}
	The GEVP 
	\begin{equation}
		\M P_n(x) = \lambda_n \L P_n(x)  \label{eq:GEVP_ML}
	\end{equation}
	is solved by the polynomials
	\begin{equation}
		P_n(x;\al,\bt,N) = {_3}F_2 \left( {-n, -x, \alpha+1 \atop -N, 1-\beta-n } ;1\right) \label{eq:Pn} 
	\end{equation}
	with eigenvalues
	\begin{equation}
		\lambda_n=n, \label{eq:lambda}
	\end{equation}
	for $n=0,1,\dots,N$.
\end{prop}
\begin{proof}
	First, the eigenvalues $\lambda_n$ in \eqref{eq:GEVP_ML} are given by the roots of the determinant $|M-\lambda L|$, which is a polynomial of degree $N+1$ in $\lambda$. From the bidiagonal form of the operators $\L$ and $\M$ given in \eqref{eq:L_mat} and \eqref{eq:M_mat}, one finds that the possible values for $\lambda_n$ are
	\begin{equation}
		\lambda_n = \frac{\eta^{(3)}_n}{\eta^{(1)}_n} = n, \quad \text{for } n=0,1,\dots,N,
	\end{equation}
	which shows \eqref{eq:lambda}. The associated eigenfunctions $P_n(x)$ can be expanded in the basis \eqref{eq:basis_phi} as
	\begin{equation}
		P_n(x) = \sum_{k=0}^N c_{n,k} \vp_k(x) \label{eq:exp_Pvp}
	\end{equation}
	for some coefficients $c_{n,k}$ to be determined. Using \eqref{eq:lambda}, \eqref{eq:exp_Pvp} as well as the actions \eqref{eq:L_phi} and \eqref{eq:M_phi} of the operators $\L$ and $\M$, one finds from the GEVP \eqref{eq:GEVP_ML} that the coefficients $c_{n,k}$ satisfy the two-term recurrence relation
	\begin{equation}
		(k+1)(k-N)(k+1-\beta-n)c_{n,k+1} = (k-n)(k+\al+1)c_{n,k}, \quad \text{for } k=0,1,\dots,N,
		\label{eq:rec_W}
	\end{equation}
	with $c_{n,N+1}:=0$.
	The solution to this recurrence relation is 
	\begin{equation}
		c_{n,k} = \frac{(-n)_k(\al+1)_k}{k!(-N)_k(1-\beta-n)_k}c_{n,0}, \qquad \text{for } k=1,\dots,N. \label{eq:c_nk}
	\end{equation}
	We will chose for convenience $c_{n,0}=1$.
	Substituting \eqref{eq:c_nk} and \eqref{eq:basis_phi} in \eqref{eq:exp_Pvp}, one arrives at the solution \eqref{eq:Pn}. 
\end{proof}

The polynomials $P_n(x;\al,\bt,N)$ will be the object of study of the remaining of this paper. Let us remark that they correspond to a discrete version of the Askey polynomials \cite{Ask}. Indeed, by performing the replacement $x \to Nx$ in the explicit expression \eqref{eq:Pn} and then taking the limit $N\to \infty$, one recovers the continuous Askey polynomials:
\begin{equation}
	\lim_{N \to \infty }P_n(Nx;\alpha,\beta,N) =
		{_2}F_1 \left( {-n, \alpha+1 \atop 1-\beta-n } ;x\right) =: \Tilde{P}_n(x;\alpha,\beta). \label{eq:limAskpol}
\end{equation}
The polynomials $P_n(x;\al,\bt,N)$ will be called $R_I$ polynomials of Hahn type because of the recurrence relation that they satisfy, see Section \ref{sec:bispect}. Note that we write $P_n(x)=P_n(x;\alpha,\beta,N)$ for simplicity when the explicit dependency of the polynomials on the parameters is not necessary.

\section{Biorthogonality}\label{sec:biorth}
We now obtain the biorthogonality property of the $R_I$ polynomials $P_n(x;\al,\bt,N)$ using the adjoint of the GEVP considered in the previous section.

Let us define a scalar product on the space of rational functions on the linear grid $x=0,1,\dots,N$ depending on some parameters $\al,\bt$ by
\begin{equation}
	(f(x),g(x))_{(\al,\bt)} =  \sum_{x=0}^N w_x{(\al,\bt)} f(x)g(x), \label{eq:scalarprod}
\end{equation}
with the weight function $w_x{(\al,\bt)}$ given by
\begin{equation}
	w_x{(\al,\bt)}= \frac{(-\alpha-\beta-N)_N}{(-\alpha-N)_N}\frac{(-N)_x (-\beta)_x}{x! (-\alpha-\beta-N)_x}. \label{eq:w_x}
\end{equation}
Note that the normalization for this weight function is chosen such that
\begin{equation}
	\sum_{x=0}^N w_x{(\alpha,\beta)} = 1,
\end{equation}
as can be verified using the Chu--Vandermonde formula ${_2}F_1 (-n, b;c;1 ) = (c-b)_n/(c)_n$, see for \textit{e.g.}\ \cite{Koek}. If $X$ is an operator acting on this space of rational functions, then its adjoint is defined as the operator $X^*$ that satisfies
\begin{equation}
	(f(x),Xg(x))_{(\al,\bt)}=(X^*f(x),g(x))_{(\al,\bt)}.
\end{equation}

We can now consider the adjoint of the GEVP \eqref{eq:GEVP_ML}:
\begin{equation}
	M^* P^*_n(x) = n L^* P^*_n(x). \label{eq:GEVP_MLad}
\end{equation}
The notation $P^*_n(x)$ refers to the functions that solve the adjoint GEVP. The interest of considering \eqref{eq:GEVP_MLad} is given by the following known result that we state without reproducing the proof.
\begin{prop}\label{prop:biorth} \textup{\cite{Zhe99}} 
	The functions
	\begin{equation}
		V_n(x):=L^*P_n^*(x) \label{eq:V_LP}
	\end{equation}
	are biorthogonal partners for $P_n(x)$ with respect to the scalar product \eqref{eq:scalarprod}, that is we have
	\begin{equation}
		\sum_{x=0}^N w_x P_n(x) V_m(x) = h_n \, \delta_{nm}, \label{eq:biort2}
	\end{equation}
	where $h_n$ is a normalization constant.
\end{prop} 
We now obtain the explicit expression for the functions $V_n(x)$.
\begin{prop}\label{prop:V_n}
	The biorthogonal partners $V_n(x)$ of the $R_I$ polynomials of Hahn type $P_n(x)$ are given by the rational functions
	\begin{equation}
		V_n(x;\al,\bt,N) = \: {_3}F_2 \left( {-n, -x, -\alpha-N \atop -N, 1+\beta-x } ;1\right). \label{eq:V_n} 
	\end{equation}
\end{prop}
\begin{proof}
	Using the scalar product \eqref{eq:scalarprod} and the weight function \eqref{eq:w_x}, one finds the following expressions for the adjoint shift operators
	\begin{align*}
		&(T^+)^* = \frac{w_{x-1}}{w_x} T^- = \frac{x(\alpha+\beta+N-x+1)}{(-N+x-1)(\beta-x+1)} T^-,\\
		&(T^-)^* = \frac{w_{x+1}}{w_x} T^+ = \frac{(-N+x)(\beta-x)}{(x+1)(\alpha+\beta+N-x)} T^+.
	\end{align*} 
	This allows to compute the adjoint operators of $\L$ and $\M$ defined in \eqref{eq:L_op} and \eqref{eq:M_op}:
	\begin{align}
		&L^* =  -\frac{x(1+\alpha+\beta+N-x)}{(1+\beta-x)} T^-+(x-N-\alpha-1)\cI, \label{eq:Ls_op}\\
		&M^* = \frac{x(\beta-x)(1+\alpha+\beta+N-x)}{(1+\beta-x)} T^- + (N-x)(x-\beta)T^+ \label{eq:Ms_op}\\
		& \qquad + \{2x^2-(2N+2\beta+\alpha)x+(\beta-1)N\} \cI. \nonumber
	\end{align}
	A basis for the rational functions on the linear grid is 
	\begin{equation}
		\rho_n(x)= \frac{(-x)_n}{(1+\beta-x)_n}. \label{eq:basis_rho}
	\end{equation}
	The adjoint operators $\L^*$ and $\M^*$ given in \eqref{eq:Ls_op} and \eqref{eq:Ms_op} act on this basis as
	\begin{align}
		&L^* \rho_n(x) = \chi_n^{(1)} \rho_{n+1}(x) + \chi_n^{(2)}\rho_n(x), \label{eq:Ls_rho}\\
		&M^* \rho_n(x) = \chi_n^{(3)} \rho_{n+1}(x) +\chi_n^{(4)} \rho_n(x) + \chi_n^{(5)} \rho_{n-1}(x), \label{eq:Ms_rho}
	\end{align}
	where
	\begin{alignat}{3}
		&\chi_n^{(1)} = \alpha+N-n, \quad &&\chi_n^{(2)} = -(1+\alpha+N-n), && \\
		&\chi_n^{(3)} = (\alpha+N-n)(1+n), \quad &&\chi_n^{(4)} = -N(2n+1)+n(2n-\alpha), \quad &&\chi_n^{(5)} = n(1+N-n).
	\end{alignat}

	One can verify that the operator $\M^*$ factorizes as
	\begin{equation}
		M^*=Z^*L^*,
	\end{equation}
	where $Z^*$ is an operator that acts in a bidiagonal fashion on the basis \eqref{eq:basis_rho}: 
	\begin{align}
		Z^* \rho_n(x) = n\rho_n(x)-\frac{n(1+N-n)}{(1+\alpha+N-n)}\rho_{n-1}(x). \label{eq:Zs_rho}
	\end{align}
	Therefore, the adjoint GEVP \eqref{eq:GEVP_MLad} becomes the following ordinary eigenvalue problem
	\begin{equation}
		Z^*V_n(x) = n V_n(x), \label{eq:EVPV}
	\end{equation}
	where we have used the definition \eqref{eq:V_LP}.
	
	To determine the explicit expression of the functions $V_n(x)$, we write
	\begin{equation}
		V_n(x) = \sum_{k=0}^N d_{n,k} \rho_k(x), \label{eq:decompV}
	\end{equation}
	for some coefficients $d_{n,k}$.
	Substituting \eqref{eq:decompV} in \eqref{eq:EVPV} and then using the action \eqref{eq:Zs_rho}, one gets the two-term recurrence relation
	\begin{equation}
		(k-n) d_{n,k} -\frac{(k+1)(k-N)}{(k-\alpha-N)} d_{n,k+1}=0 \quad \text{for } k=0,1,\dots,N
	\end{equation}
	with $d_{n,N+1}:=0$. The solution to this recurrence relation is given by
	\begin{equation}
		d_{n,k} = \frac{(-n)_k(-\alpha-N)_k}{k!(-N)_k} d_{n,0}, \quad \text{for } k=1,\dots,N. \label{eq:coeffV}
	\end{equation}
	We chose again for convenience $d_{n,0}=1$.
	Combining \eqref{eq:coeffV}, \eqref{eq:decompV} and \eqref{eq:basis_rho}, one obtains the result \eqref{eq:V_n}.
\end{proof}
The normalization factor $h_n$ will be obtained later in Section \ref{sec:HahnOP}, see Proposition \ref{prop:h_n}.

\section{Bispectrality}\label{sec:bispect}
In this section, we obtain the difference and recurrence relations of the $R_I$ polynomials of Hahn type.  

The triplet of operators $\L,\M,\Y$ act on the polynomials $P_n(x)$ by shifting their parameters as follows 
\begin{align}
	&\L P_n(x;\alpha,\beta,N) = -\frac{(\alpha+1)(1-\beta)}{(1-\beta-n)} P_n(x;\alpha+1,\beta-1,N), \label{eq:L_Pn_shift} \\
	&\M P_n(x;\alpha,\beta,N) = -n\frac{(\alpha+1)(1-\beta)}{(1-\beta-n)} P_n(x;\alpha+1,\beta-1,N), \label{eq:M_Pn_shift} \\
	&\Y P_n(x;\alpha,\beta,N) = -x\frac{(\alpha+1)(1-\beta)}{(1-\beta-n)} P_n(x;\alpha+1,\beta-1,N). \label{eq:Y_Pn_shift}
\end{align}
The previous equations can indeed be derived directly by using the definitions \eqref{eq:L_op}--\eqref{eq:Y_op} of the operators and the explicit expression \eqref{eq:Pn} of the $R_I$ Hahn polynomials. Note that equations \eqref{eq:L_Pn_shift} and \eqref{eq:M_Pn_shift} are consistent with the GEVP \eqref{eq:GEVP_ML}, while equation \eqref{eq:Y_Pn_shift} simply reflects the equality observed in \eqref{eq:YxL}.   

\subsection{Recurrence relation}

One can verify (using contiguity relations for ${_3}F_2$ functions for instance) that the polynomials with shifted parameters that appear in \eqref{eq:L_Pn_shift}--\eqref{eq:Y_Pn_shift} satisfy
\begin{equation}
	\frac{(\alpha+1)(1-\beta)}{(1-\beta-n)} P_n(x;\alpha+1,\beta-1,N) =
	(n+\alpha+1)P_n(x;\alpha,\beta,N) + 
	\frac{n(\alpha+\beta+n)}{(1-\beta-n)}P_{n-1}(x;\alpha,\beta,N). \label{eq:Pn_shift}
\end{equation}
It is apparent from \eqref{eq:L_Pn_shift} and \eqref{eq:Pn_shift} that the action of the operator $\L$ on the basis of polynomials $P_n(x)=P_n(x;\alpha,\beta,N)$ for $n=0,1,\dots,N$ is
\begin{align}
	L P_n(x) = - (n+\alpha+1)P_n(x) -\frac{n(\alpha+\beta+n)}{(1-\beta-n)} P_{n-1}(x). \label{eq:L_Pn}
\end{align}
A straightforward computation performed using the expansion \eqref{eq:exp_Pvp} of the polynomials $P_n(x)$ in the Pochhammer basis $\vp_n(x)$ and the result \eqref{eq:Pn_shift} also shows that the action of the operator $\Y$ on the basis of polynomials $P_n(x)$ is
\begin{align}
	&Y P_n(x) = \cA_n P_{n+1}(x) - \left(  \cA_n+\cB_n \right) P_n(x) +\cB_n P_{n-1}(x), \label{eq:Y_Pn}
\end{align}
where
\begin{align}
	\cA_n = (n-N)(\beta+n), \quad \cB_n=n(\alpha+\beta+n).
\end{align}
Applying both sides of equality \eqref{eq:YxL} $\Y = x \L$ on $P_n(x)$ and using the tridiagonal actions \eqref{eq:L_Pn} and \eqref{eq:Y_Pn} of the operators $\L$ and $\Y$ to re-express each side, one gets the three-term recurrence relation for the polynomials $P_n(x)$:
\begin{align}
	&\cA_n P_{n+1}(x) -(\cA_n+\cB_n)P_n(x) + \cB_n P_{n-1}(x) \nonumber \\
	 &= -x\left((n+\alpha+1) P_n(x) + \frac{n(\alpha+\beta+n)}{(1-\beta-n)}P_{n-1}(x)\right).
\end{align}

\begin{rem}
	It is standard to renormalize the polynomials $P_n(x)$ to make them monic (\textit{i.e.}\ with leading term $x^n$):
	\begin{equation}
		p_n(x):=\mu_n P_n(x) \quad \text{with} \quad \mu_n=\frac{(-N)_n(1-\bt-n)_n}{(\al+1)_n}.
	\end{equation}
	The normalized recurrence relation then reads
	\begin{align}
		p_{n+1}(x) +\left(\gamma_n -x \right)p_n(x) + \delta_n (x-\epsilon_n) p_{n-1}(x) =0, \label{eq:recmonic}
	\end{align}
	where we have defined
	\begin{equation}
		\gamma_n = \frac{2 n^2 + (\al+2\bt-N)n -\bt N }{n+\al+1}, \quad \delta_n = \frac{n(\al+\bt+n)(N+1-n)}{(\al+n)(\al+1+n)}, \quad \epsilon_n =n+\bt-1. \label{eq:recconstmonic}
	\end{equation}
	The recurrence relation \eqref{eq:recmonic} with a linear term $x-\epsilon_n$ in front of $p_{n-1}$ is said to be of $R_I$ type \cite{IM}.
\end{rem}

\begin{rem}
	Let us denote
	\begin{equation}
		\Tilde{p}_n(x) := (-1)^n\frac{(1-\bt-n)_n}{(\al+1)_n} \Tilde{P}_n(x;\al,\bt),
	\end{equation}
	where we recall that $\Tilde{P}_n(x;\al,\bt)$ are the continuous Askey polynomials recovered in \eqref{eq:limAskpol} from the $R_I$ polynomials $P_n(x;\al,\bt,N)$ by replacing $x\to Nx$ and taking the limit $N \to \infty$. Performing the same replacement and limit in the normalized recurrence relation \eqref{eq:recmonic}, one finds
	\begin{equation}
		\Tilde{p}_{n+1}(x) - \left( \frac{n+\bt}{n+\al+1} + x \right) \Tilde{p}_n(x) + x \frac{n(n+\al+\bt)}{(n+\al)(n+\al+1)} \Tilde{p}_{n-1}(x) =0. \label{eq:recAsk}
	\end{equation}
	Equation \eqref{eq:recAsk} corresponds to the recurrence relation of the Askey polynomials obtained in \cite{VZ_Ask,HR}.
\end{rem}

\subsection{Difference equation} 
As seen and used already, the GEVP \eqref{eq:GEVP_ML} amounts to a difference equation when $\L$ and $\M$ are the shift operators introduced from the beginning in \eqref{eq:L_op} and \eqref{eq:M_op}. It reads
\begin{equation}
	A(x)P_n(x+1) + B(x)P_n(x-1) + C(x) P_n(x)  = n\left((N-x)P_n(x+1) +(x-N-\al-1)P_n(x)\right),
\end{equation} 
where the coefficients $A(x),B(x),C(x)$ are given in \eqref{eq:ABC}.

\section{Connection with Hahn orthogonal polynomials}\label{sec:HahnOP}
In this section, we identify a connection between the $R_I$ Hahn biorthogonal polynomials and the Hahn orthogonal polynomials. We then use this connection in order to obtain the normalization constant $h_n$ of the biorthogonality relation \eqref{eq:biort2}. 

The Hahn polynomials form a family of orthogonal polynomials defined on the linear grid $x=0,1,\dots,N$, for $N$ a non-negative integer, see \cite{Koek}. They are given in terms of hypergeometric series by 
\begin{equation}
	H_n(x;\xi,\eta,N) =  {_3}F_2 \left(  { -n, -x, n+\xi+ \eta +1 \atop -N, \xi+1 }; 1\right),
	\label{eq:Hahn_n}
\end{equation}
where $\xi$ and $\eta$ are two real parameters and $n=0,1,\dots,N$. They satisfy the orthogonality relation
\begin{equation}
	\sum_{x=0}^N w_x^{(H)} H_n(x) H_m(x) = h_n^{(H)} \, \delta_{nm} \label{eq:orthHahn}
\end{equation}
with weight function 
\begin{equation}
	w_x^{(H)}(\xi,\eta,N)= \frac{(1+\eta)_N}{N!}\frac{(-N)_x (1+\xi)_x}{x! (-\eta-N)_x} \label{eq:wH_x}
\end{equation}
and normalization constant
\begin{equation}
	h_n^{(H)}(\xi,\eta,N) = \frac{(-1)^n(n+\xi+\eta+1)_{N+1}(\eta+1)_n \ n!}{(2n+\xi+\eta+1)(\xi+1)_n(-N)_n \ N!}. \label{eq:norm_Hahn}
\end{equation}

Let us now take
\begin{equation}
	\xi=-\beta-n, \qquad \eta = \alpha + \beta, \label{eq:param}
\end{equation}
where $\alpha,\beta$ are real parameters\footnote{The first reparametrization was referred to as gluing in \cite{KS} and used to obtain $R_I$-polynomials without obtaining the biorthogonality relation nor their bispectral properties.}. Applying this choice to expression \eqref{eq:Hahn_n}, one recovers the $R_I$ Hahn polynomials given in \eqref{eq:Pn}, that is
\begin{equation}
	H_n(x;-\beta-n,\alpha + \beta,N) = P_n(x;\alpha,\beta,N). \label{eq:Hparam}
\end{equation} 
Substituting the choice of parameters \eqref{eq:param} in the weight function \eqref{eq:wH_x} and using the identity 
\begin{equation}
	(a-m)_k= \frac{(1-a)_m(a)_k}{(1-a-k)_m},
\end{equation}
one finds an expression which now depends on $n$ and which involves the weight function of the $R_I$ Hahn polynomials:
\begin{align}
	w_x^{(H)}(-\beta-n,\alpha + \beta,N) 
	&=  \frac{(\alpha+1)_N(1+\beta)_{n-1}}{N! \ (1+\beta-x)_{n-1}}w_x(\alpha,\beta), \label{eq:wHparam}
\end{align}
where we have used the equality
\begin{equation}
	\frac{(-\alpha-\beta-N)_N}{(-\alpha-N)_N} = \frac{(\alpha+\beta+1)_N}{(\alpha+1)_N}. 
\end{equation}
The orthogonality relation \eqref{eq:orthHahn} implies that for any polynomial $\pi_m(x)$ in $x$ of degree $m$ defined on the linear grid we have
\begin{equation}
	\sum_{x=0}^N w_x^{(H)} H_n(x)\pi_m(x) = 0 \quad \text{for } 0\leq m<n.  \label{eq:orthHahn2}
\end{equation}
It is true in particular for the polynomial
\begin{equation}
	\pi_m(x) = \frac{(1+\beta-x)_{n-1}}{(1+\beta-x)_{n-1-m}} = (\epsilon_n-x)(\epsilon_{n-1}-x)\dots(\epsilon_{n-m+1}-x),
\end{equation}
where $\epsilon_n$ is the same constant as the one defined in \eqref{eq:recconstmonic}. 
Therefore, by taking the parameters $\xi$ and $\eta$ as in \eqref{eq:param} in equation \eqref{eq:orthHahn2} and using the results \eqref{eq:Hparam} and \eqref{eq:wHparam}, one gets
\begin{equation}
	\sum_{x=0}^N w_x(\alpha,\beta)\frac{1}{(1+\beta-x)_{m}} P_n(x;\alpha,\beta,N) = 0 \quad \text{for } 0 \leq m < n. \label{eq:biorth3}
\end{equation} 
The previous equation implies the biorthogonality relation \eqref{eq:biort2} for $m<n$. Indeed, one can express the basis functions $\rho_n(x)$ defined in \eqref{eq:basis_rho} as a linear combination of the rational monomials
\begin{equation}
	\frac{1}{(1+\bt-x)_k}. \label{eq:ratmonom}
\end{equation}
More precisely 
\begin{equation}
	\rho_n(x)=\frac{(-x)_n}{(1+\beta-x)_n} = \sum_{k=0}^n \frac{u_k}{(1+\beta-x)_n}, \label{eq:rhoratmonom}
\end{equation}
where $u_k$ are coefficients that can be determined. For instance, one finds
\begin{equation}
	u_n = (-\beta-n)_n. \label{eq:u_n}
\end{equation}
It follows that the rational functions $V_m(x)$ obtained in \eqref{eq:V_n} can be expanded on the basis \eqref{eq:ratmonom} with highest degree in the denominator $k=m$. For $m<n$ all the terms in this expansion are orthogonal to $P_n(x)$ because of \eqref{eq:biorth3}, therefore implying the biorthogonality relation \eqref{eq:biort2} for $m<n$.

The connection with the Hahn orthogonal polynomials allows to prove the following result.
\begin{prop}\label{prop:h_n}
	The normalization factor in the biorthogonality relation \eqref{eq:biort2} is given by
	\begin{equation}
		h_n = \frac{n!(1+\al + \beta)_n}{(-N)_n(\beta)_n}. \label{eq:h_n}
	\end{equation}
\end{prop}
\begin{proof}
	Let us define the factor
	\begin{equation}
		\kappa_n = \frac{(-N)_n(\xi+1)_n}{(n+\xi+\eta+1)_n} \label{eq:kap_n}
	\end{equation}
	which ensures a monic normalization for the Hahn polynomials \eqref{eq:Hahn_n}. By Christoffel transform, 
	\begin{equation}
		H_n(x;\xi,\eta,N) = \frac{\kappa_{n+1}(\xi-1)}{\kappa_n(\xi)} \frac{H_{n+1}(x;\xi-1,\eta,N)-W_n H_{n}(x;\xi-1,\eta,N)}{x+\xi}, \label{eq:Chr}
	\end{equation}
	where we have explicitly indicated the dependance of $\kappa_n$ on the parameter $\xi$ and where
	\begin{equation}
		W_n = \frac{H_{n+1}(-\xi;\xi-1,\eta,N)}{H_{n}(-\xi;\xi-1,\eta,N)}.
	\end{equation}
	One can also use the Christoffel-Darboux formula to write 
	\begin{equation}
		\frac{H_{n+1}(x;\xi-1,\eta,N)-W_n H_{n}(x;\xi-1,\eta,N)}{x+\xi} = \frac{\kappa_{n}(\xi-1)}{\kappa_{n+1}(\xi-1)} \sum_{k=0}^{n} Y_{nk} H_{k}(x;\xi-1,\eta,N), \label{eq:ChrDar}
	\end{equation}
	where
	\begin{equation}
		Y_{nk} = \frac{h_n^{(H)}(\xi-1) H_{k}(-\xi;\xi-1,\eta,N)}{h_k^{(H)}(\xi-1)H_{n}(-\xi;\xi-1,\eta,N)}. \label{eq:Y_nk}
	\end{equation}
	Combining \eqref{eq:Chr} and \eqref{eq:ChrDar}, one thus finds
	\begin{equation}
		H_n(x;\xi,\eta,N) = \frac{\kappa_{n}(\xi-1)}{\kappa_n(\xi)} \sum_{k=0}^{n} Y_{nk} H_{k}(x;\xi-1,\eta,N).
	\end{equation}
	It directly follows that
	\begin{equation}
		\sum_{x=0}^N w_x^{(H)}(\xi-1) H_n(x;\xi,\eta,N) = \frac{\kappa_{n}(\xi-1)}{\kappa_n(\xi)} \sum_{k=0}^{n} Y_{nk}  \sum_{x=0}^N w_x^{(H)}(\xi-1) H_{k}(x;\xi-1,\eta,N). \label{eq:renorm1}
	\end{equation}
	The case $m=0$ in the orthogonality relation \eqref{eq:orthHahn} of the Hahn polynomials can be used to simplify the RHS of \eqref{eq:renorm1}. As a result,
	\begin{equation}
		\sum_{x=0}^N w_x^{(H)}(\xi-1) H_n(x;\xi,\eta,N) = \frac{\kappa_{n}(\xi-1)}{\kappa_n(\xi)} Y_{n0}  h_0^{(H)}(\xi-1).
	\end{equation}
	Using the explicit expression for $Y_{nk}$ given in \eqref{eq:Y_nk}, one finally gets
	\begin{equation}
		\sum_{x=0}^N w_x^{(H)}(\xi-1) H_n(x;\xi,\eta,N) = \frac{\kappa_{n}(\xi-1)}{\kappa_n(\xi)} \frac{h_n^{(H)}(\xi-1)}{H_{n}(-\xi;\xi-1,\eta,N)}. \label{eq:renorm2}
	\end{equation}
	With the choice of parameters \eqref{eq:param}, the LHS of \eqref{eq:renorm2} becomes 
	\begin{equation}
		\sum_{x=0}^N w_x^{(H)}(\xi-1) H_n(x;\xi,\eta,N) = \frac{(1+\alpha)_N(1+\beta)_{n}}{N!} \sum_{x=0}^N w_x(\alpha,\beta)\frac{1}{(1+\beta-x)_{n}} P_n(x;\alpha,\beta,N).
	\end{equation}
	%
	Under the same choice of parameters \eqref{eq:param}, it is found that
	\begin{equation}
		H_{n}(-\xi;\xi-1,\eta,N) = {_2}F_1 \left(  { -n, \alpha \atop -N }; 1\right) = \frac{(-N-\alpha)_n}{(-N)_n}, \label{eq:renorm3}
	\end{equation}
	where we have used the Chu--Vandermonde formula. The RHS of \eqref{eq:renorm2} with parameters \eqref{eq:param} can hence be computed explicitly using \eqref{eq:norm_Hahn}, \eqref{eq:kap_n} and \eqref{eq:renorm3}. After some simplifications, equation \eqref{eq:renorm2} becomes
	\begin{equation}
		\sum_{x=0}^N w_x(\alpha,\beta)\frac{1}{(1+\beta-x)_{n}} P_n(x;\alpha,\beta,N) = \frac{(-1)^n(\alpha+\beta+1)_n \ n!}{(\beta)_{n}(-N-\alpha)_n(-\beta-n)_n}. \label{eq:renorm4}
	\end{equation}	
	Using the explicit expression \eqref{eq:V_n} for $V_n(x)$, the properties \eqref{eq:rhoratmonom} and \eqref{eq:u_n} of the rational basis functions $\rho_n(x)$ as well as the biorthogonality relation \eqref{eq:biorth3}, it follows that
	\begin{align}
		&\sum_{x=0}^N w_x(\al,\bt)V_n(x;\al,\bt,N) P_n(x;\al,\bt,N) \nonumber \\ 
		&=\frac{(-n)_n(-\alpha-N)_n}{(-N)_n \ n!} \sum_{x=0}^N w_x(\alpha,\beta)\frac{(-\beta-n)_n}{(1+\beta-x)_{n}} P_n(x;\alpha,\beta,N). \label{eq:renorm5}
	\end{align}
	It is then straightforward to use \eqref{eq:renorm4} in the RHS of \eqref{eq:renorm5} to obtain
	\begin{align}
		\sum_{x=0}^N w_x(\al,\bt)V_n(x;\al,\bt,N) P_n(x;\al,\bt,N) 
		&= \frac{n!(\alpha+\beta+1)_n }{(-N)_n(\beta)_{n}}. \label{eq:renorm6}
	\end{align}
	Comparing \eqref{eq:renorm6} with \eqref{eq:biort2}, the result \eqref{eq:h_n} follows.
\end{proof}

\section{Biorthogonality relation as a limit}\label{sec:biorthlim}
A biorthogonality relation involving rational $_{10} \phi_9$ basic hypergeometric functions which are related to Wilson's system of biorthogonal functions \cite{Wil} has been obtained in \cite{GM}. In the present section, we establish a connecton between these $_{10} \phi_9$ functions and the $R_I$ polynomials of Hahn type. Note that we use the standard definition for the basic hypergeometric functions $_{r} \phi_s$, see \cite{GR}. We also use the following notations for the $q$-Pochhammer symbols:
\begin{align}
	&(a;q)_0:=1, \quad (a;q)_n:=\prod_{i=0}^{n-1}(1-aq^i), \quad n=1,2,\dots\\
	&(a_1,a_2,\dots,a_k;q)_n:=\prod_{i=1}^k(a_i;q)_n \quad n=0,1,2,\dots
\end{align}     
  
The general biorthogonality relation provided in Theorem 4.1 in \cite{GM} admits limiting cases. Of particular interest for our purposes is the case given in Corollary 4.2 in \cite{GM} involving $_{4} \phi_3$ functions. With different notations, the result is that the functions
\begin{align}
	&\widetilde{U}_n(x) = \frac{(q^{a+1},q^{a+1-d}/e;q)_n}{(q^{a+1-d},q^{a+1}/e;q)_n} {_4}\phi_3 \left( {q^{-n},q^{-x}, q^d,e \atop q^{-N}, q^{N-x+a+1},q^{d-n-a}e} ; q,q \right),\\
	&\widetilde{V}_m(x) = \frac{(q^{a+N},q^{a-N+2-d}/e;q)_m}{(q^{a+1-d},q^{a+1}/e;q)_m} {_4}\phi_3 \left( {q^{-m},q^{-x}, q^{-N+1-d},q^{-N+1}/e \atop q^{-N}, q^{-N-m+1-a}, q^{-x+a+2-b-d}/e} ; q,q \right),
\end{align}
satisfy the biorthogonality relation
\begin{equation}
	\sum_{x=0}^N \wt{w}_x \wt{U}_n(x) \widetilde{V}_m(x) = \wt{h}_n \delta_{nm}, \label{eq:biortWilson}
\end{equation}
where the weight function and normalization constant are given by
\begin{align}
	&\wt{w}_x = q^{N-x} \frac{(q^{a+1-d},q^{-N},q^{a+1}/e;q)_{N-x}}{(q^{a+1},q,q^{a+2-N-d}/e;q)_{N-x}}\frac{(q^{a+1},q^{d-a-1}e;q)_N}{(q^d,e;q)_N},\\
	&\wt{h}_n = q^{-n} \frac{(q,q^{a+1},q^{a+2-N-d}/e;q)_{n}}{(q^{-N},q^{a+1-d},q^{a+1}/e;q)_{n}}.
\end{align}
(For convenience, we have modified the parameters and variables of \cite{GM} as follows: $a\to q^a, d\to q^d, k\to N-x$.) The biorthogonality relation of the $R_I$ polynomials of Hahn type can be recovered after taking two limits as we explain next. 

Let us start by taking the limit $e \to \infty$. One finds 
\begin{align}
	&\lim_{e\to\infty}\widetilde{U}_n(x) = \frac{(q^{a+1};q)_n}{(q^{a+1-d};q)_n} {_3}\phi_2 \left( {q^{-n},q^{-x}, q^d \atop q^{-N}, q^{N-x+a+1}} ; q, q^{1+n+a-d} \right),\\
	&\lim_{e\to\infty}\widetilde{V}_m(x) = \frac{(q^{a+N};q)_m}{(q^{a+1-d};q)_m} {_3}\phi_2 \left( {q^{-m},q^{-x}, q^{-N+1-d} \atop q^{-N}, q^{-N-m+1-a}} ; q,q \right), \\
	&\lim_{e\to\infty} \wt{w}_x = \frac{(q^{a+1-d},q^{-N};q)_{N}}{(q^{d},q;q)_{N}}\frac{(q^{-N-a},q^{-N};q)_x}{(q^{-N-a+d},q;q)_x}q^{N(d-a)+x(N+d)}, \label{eq:wWils2}\\
	&\lim_{e\to\infty} \wt{h}_n = q^{-n} \frac{(q,q^{a+1};q)_{n}}{(q^{-N},q^{a+1-d};q)_{n}}.
\end{align}
The following identity \cite{GR} has been used in \eqref{eq:wWils2}: 
\begin{equation}
	(z;q)_{n-k} = \frac{(z;q)_n}{(q^{1-n}/z;q)_k}\left(-\frac{q}{z}\right)^k q^{\binom{k}{2}-nk}.
\end{equation}
Then, we take the limit $q\to 1$ of the previous results:
\begin{align}
	&\lim_{q\to 1} \lim_{e\to\infty} \widetilde{U}_n(x) = \frac{(a+1)_n}{(a+1-d)_n} {_3}F_2 \left( {-n,-x, d \atop -N, N-x+a+1} ; 1 \right), \label{eq:UWils3}\\
	&\lim_{q\to 1} \lim_{e\to\infty} \widetilde{V}_m(x) = \frac{(a+N)_m}{(a+1-d)_m} {_3}F_2 \left( {-m,-x, -N+1-d \atop -N, -N-m+1-a} ; 1 \right), \label{eq:VWils3} \\
	&\lim_{q\to 1} \lim_{e\to\infty} \wt{w}_x = \frac{(a+1-d)_N(-N)_{N}}{N!(d)_N}\frac{(-N-a)_x(-N)_x}{x!(-N-a+d)_x},\\
	&\lim_{q\to 1} \lim_{e\to\infty} \wt{h}_n =  \frac{n!(a+1)_{n}}{(-N)_n(a+1-d)_{n}}.
\end{align}
It is apparent now that with the choice of parameters 
\begin{equation}
	a=\beta -N, \qquad d=-\alpha-N, \label{eq:paramWils}
\end{equation}
the expressions \eqref{eq:UWils3} and \eqref{eq:VWils3} are respectively equal up to a normalization factor to $V_n(x)$ and $P_m(x)$, which are given in \eqref{eq:V_n} and \eqref{eq:Pn}. It is then straightforward to verify using the previous results that in the limit where $e \to \infty$ and $q \to 1$, and with parameters as in \eqref{eq:paramWils}, the biorthogonality relation \eqref{eq:biort2} with normalization constant \eqref{eq:h_n} is recovered from \eqref{eq:biortWilson}.

Note that this connection between the $_{10} \phi_9$ rational biorthogonal functions of \cite{GM} and the $R_I$ polynomials of Hahn type provides an alternative method for proving the expression \eqref{eq:h_n} of the normalization constant $h_n$.  

\section{Conclusion}\label{sec:conc}

Summing up, we have obtained the $R_I$ polynomials of Hahn type $P_n(x)$ as solutions of a GEVP involving two difference operators $\M$ and $\L$. We have then found the rational biorthogonal partners $V_n(x)$ of these polynomials by solving the adjoint GEVP. The complete biorthogonality relation was obtained using the connections with the Hahn orthogonal polynomials and was also recovered as a limit of a biorthogonality relation implying ${}_4\phi_3$ basic hypergeometric functions. We have shown that the biorthogonal polynomials $P_n(x)$ obey a three-term recurrence relation of $R_I$ type, by examining the action of $\L$ as well as that of a third difference operator $\Y$ on the basis of polynomials $P_n(x)$. The action of the pair $\M$ and $\L$ had led to a difference relation. Therefore, the triplet of operators $\L,\M,\Y$ offers a description of the bispectrality of the $R_I$ biorthogonal polynomials of Hahn type.      

This study is raising interesting questions that would deserve exploration. One of them has to do with the characterization of the algebra formed by the triplet of bispectral operators and with the determination of its representation theory. In the case of orthogonal polynomials, the polynomials appear as overlap coefficients between the eigenbases of the two bispectral operators. Similar story occurs for the biorthogonal functions with bispectral properties, see \cite{VZ_Hahn} for the Hahn rational functions and \cite{VZ_Ask} for the Askey polynomials on the unit circle. In the first case (resp. second case), a more primitive algebra called meta-Hahn algebra (resp. meta-Jacobi algebra) is shown to encompass both the bispectral properties of the Hahn (resp. Askey) biorthogonal functions as well as the Hahn (resp. Jacobi) orthogonal polynomials. A meta-$q$-Hahn algebra was also defined in \cite{BGVZ} for the rational functions of $q$-Hahn type. It would hence be of significant interest to find a meta algebra that would describe the $R_I$ polynomials of Hahn type. Another question that should be examined in more details concerns the connections with the Askey polynomials which are biorthogonal on the unit circle \cite{VZ_Ask}. Indeed, it was observed in Section \ref{sec:GEVP} that the Hahn $R_I$ biorthogonal polynomials correspond to a discrete version of the Askey biorthogonal polynomials. In Section \ref{sec:bispect}, the recurrence relation of the Askey polynomials was recovered simply as a limit from that of the Hahn $R_I$ polynomials. A natural question that remains is to understand how the biorthogonality relations of both polynomials can be connected. We plan to examine theses aspects in the future.  

\paragraph{Acknowledgments.} The research of LV is supported by a Discovery Grant from the Natural Sciences and Engineering Research Council (NSERC) of Canada. MZ holds an Alexander--Graham--Bell graduate scholarship from NSERC.

\end{document}